\documentclass[12pt]{article}
\usepackage{fullpage}
\usepackage{amssymb}
\usepackage{amsmath}
\usepackage{amsthm}
\usepackage[all]{xy}
\usepackage{graphicx}
\usepackage{mathrsfs}

\newtheorem{theorem}{Theorem}[section]

\newtheorem{proposition}[theorem]{Proposition}
\newtheorem{corollary}[theorem]{Corollary}
\newtheorem{definition}[theorem]{Definition}
\newtheorem{example}[theorem]{Example}

\newtheorem{problem}{Problem}

\newcommand{\pref}[1]{Proposition~\textup{\ref{prop:#1}}}
\newcommand{\cref}[1]{Corollary~\textup{\ref{cor:#1}}}

\newcommand{\G}{\Gamma}
\newcommand{\calM}{\mathcal M}
\newcommand{\calP}{\mathcal P}
\newcommand{\calQ}{\mathcal Q}
\newcommand{\calR}{\mathcal R}

\newcommand{\calU}{\mathcal U}

\newcommand{\comment}[1]{}
\newcommand{\eps}{\epsilon}
\newcommand{\mix}{\diamond}

\newcommand{\covers}{\searrow}     
\newcommand{\fgraph}{\mathcal G}

\DeclareMathOperator{\Mon}{Mon}
\DeclareMathOperator{\Stab}{Stab}
\DeclareMathOperator{\Norm}{Norm}
\DeclareMathOperator{\rank}{rank}
\DeclareMathOperator{\Sym}{Sym}
\DeclareMathOperator{\Core}{Core}

\begin{document}

\title{Open problems on k-orbit polytopes}

\author{Gabe Cunningham \\
University of Massachusetts Boston \\
and \\ 
Daniel Pellicer \\
Centro de Ciencias Matem\'aticas\\
Universidad Nacional Aut\'onoma de M\'exico
}

\date{ \today }
\maketitle

\begin{abstract}
	We present 35 open problems on combinatorial, geometric and algebraic aspects of $k$-orbit abstract polytopes. We also present a theory of rooted polytopes that has appeared implicitly in previous work but has not been formalized before.
	

\comment{
	20B25: Finite automorphism groups of algebraic, geometric, or combinatorial structures.
	51M20: Polyhedra and polytopes; regular figures, division of spaces
	52B05: Combinatorial properties of polyhedra (number of faces etc.)
	52B15: Symmetry properties of polytopes
}

\end{abstract}

\section{Introduction}
	Abstract polytopes generalize convex polytopes, skeletal polyhedra, and tilings of surfaces
	and spaces. A regular (abstract) polytope is one that ``looks the same from every angle'': the
	automorphism group acts transitively on the flags. More generally, a $k$-orbit polytope 
	is one where the automorphism group has $k$ orbits on the flags. Regular polytopes
	have been extensively studied; see \cite{arp} for the standard reference, and see
	\cite{transitive-subgps, smallest-regular, realizations-4, modular-reduction} for a broad 
	cross-section of current advances. Two-orbit polytopes (including chiral
	polytopes) have also received a lot of attention; see \cite{two-orbit,chiral} for the basic notions and 
	\cite{finite-chiral-4-poly, chiral-constructions, chiral-ext2, chiral-atlas2, two-orbit-convex, chiral-problems}
	for recent work. Very little is yet known about $k$-orbit polytopes for $k \geq 3$.

	Most of what we have learned about regular and chiral polytopes
	has come from algebraic arguments, because there is a standard way of building a regular or
	chiral polytope from groups of a certain kind. Unfortunately, it is difficult to generalize
	these algebraic arguments to $k$-orbit polytopes with $k \geq 3$. It seems likely that
	arguments of a different flavor will be required.
	
	Our goal in this paper is to collect what is already known about $k$-orbit polytopes,
	and to pose what we think will be the important problems as research moves forward.

	We also take this opportunity to describe a theory of \emph{rooted polytopes}.
	When working with a regular polytope $\calP$, we usually pick an arbitrary flag $\Phi$ to be
	a \emph{base flag}, and we name the automorphisms of $\calP$ according to how they
	act on $\Phi$. In this case, the choice is merely a formality, since all of the flags
	look alike. With chiral polytopes, the choice of base flag actually matters. 
	Every chiral polytope $\calP$ has an \emph{enantiomorphic form} $\overline{\calP}$,
	which we think of as being the mirror image of $\calP$. We often speak 
	as if $\calP$ and its enantiomorphic form are different (but isomorphic) polytopes. In reality, 
	the only difference between $\calP$ and $\overline{\calP}$ is the choice of base flag;
	the underlying polytope (as a poset) is exactly the same. We hope to clarify
	this confusion --- and facilitate working with $k$-orbit polytopes --- by carrying the
	chosen base flag along as part of the notation. Thus, a rooted polytope is a pair
	$(\calP, \Phi)$, where $\Phi$ is a flag of $\calP$.

	We will see that many common polytope operations are better understood as operations on rooted
	polytopes. Covers, mixing, amalgamation, presentations for automorphism groups, and much more
	all depend on the choice of base flag. In fact, many of these operations are not well-defined
	(for non-regular polytopes) if we try to remove the reliance on a base flag.

\section{Background}
\subsection{Polytopes}
\label{sec:polytopes}

	Here we present background material on polytopes, mostly taken from \cite[Sec. 2A]{arp}.
	Let $\calP$ be a partially-ordered set with a strictly monotone rank function with outputs
	in $\{-1, 0, \ldots, n\}$. The elements of $\calP$ are called \emph{faces}, and an element
	of rank $i$ is an \emph{$i$-face}. \emph{Vertices}, \emph{edges}, and \emph{facets} are
	faces of rank 0, 1, and $n-1$, respectively. The maximal
	chains of $\calP$ are called \emph{flags}. We say that $\calP$ is a \emph{polytope of rank $n$}
	(or \emph{$n$-polytope}) provided it satisfies the following four conditions:
	\begin{enumerate}
	\item There is a unique minimal face $F_{-1}$ of rank $-1$, and a unique maximal face $F_n$ of rank $n$.
	\item Every flag contains $n+2$ faces (one in each rank).
	\item (Diamond condition): Whenever $F < G$ and $\rank G - \rank F = 2$, there are
	exactly two faces $H$ with $\rank H = \rank F + 1$ such that $F < H < G$.
	\item (Strong connectivity): Suppose $F < G$ and $\rank G - \rank F \geq 3$. If $F < H < G$
	and $F < H' < G$, then there is a chain
	\[ H = H_0 \leq H_1 \geq H_2 \leq H_3 \geq H_4 \leq \cdots \geq H_k = H' \]
	such that $F < H_i < G$ for each $i$.
	\end{enumerate}

	There is a unique polytope in each of the ranks $-1, 0,$ and $1$. In rank 2, each polytope has the
	same number of vertices and edges; we denote a 2-polytope by $\{p\}$, where $2 \leq p \leq \infty$
	is the number of vertices. The digon $\{2\}$ consists of two vertices and two edges, each containing both vertices. If $p$ is finite and at least 3, then $\{p\}$ is the face-lattice of a convex
	$p$-gon. The apeirogon $\{\infty\}$ is the face-lattice of the tiling of a line by line segments. Each \emph{polyhedron} (3-polytope) can be represented as a map on a surface (though not every
	map corresponds to a polytope). In ranks 4 and higher, the possibilities multiply dramatically.

	Like their convex counterparts, polytopes are built out of polytopes of lower rank. If $F < G$ are
	faces of a polytope $\calP$, then the \emph{section} $G/F$ consists of those faces $H$ such that
	$F \leq H \leq G$. Every section of a polytope is itself a polytope, whose rank is
	$\rank G - \rank F - 1$.
	When talking about a facet $F$ of a polytope, we usually have in mind the section $F/F_{-1}$ of $F$
	over the minimal face of $\calP$. We also often speak of the \emph{vertex-figure} of a polytope,
	which is a section $F_n / v$, where $F_n$ is the maximal face of $\calP$ and $v$ is a vertex.

	The \emph{trivial extension} of an $n$-polytope $\calP$ is an $(n+1)$-polytope defined
	as follows. We add two new faces to $\calP$; first we add $F_n'$ and give it the same
	incidences as $F_n$, and then we add a new maximal face $F_{n+1}$. That the result is
	always a polytope is easily checked. For example, if we extend the triangle $\{3\}$
	in this way, then we get two triangles ``glued back-to-back'', sharing the same
	vertices and edges; this gives us the polytope $\{3, 2\}$.
	
	The \emph{dual} of $\calP$, denoted $\calP^{*}$, has the same underlying set of $\calP$ but
	with the partial order reversed. In particular, the facets of $\calP^{*}$ are the
	vertices of $\calP$, and vice-versa.
	
	Two flags are \emph{adjacent} if they differ in only a single element. Flags that differ only
	in their $i$-face are said to be \emph{$i$-adjacent}.
	As a consequence of the diamond condition, every flag $\Phi$ has a unique $i$-adjacent
	flag $\Phi^i$. We extend this notation and for any word $w$ on the set $\{0,\dots,n-1\}$ we express $(\Phi^w)^i$ by $\Phi^{wi}$.
	
	Polytopes are also \emph{strongly flag-connected}: Given any two flags $\Phi$
	and $\Psi$, there is a sequence of flags
	\[ \Phi = \Phi_0, \Phi_1, \ldots, \Phi_k = \Psi \]
	such that $\Phi_i$ is adjacent to $\Phi_{i+1}$ for each $i$, and each $\Phi_i$ contains
	$\Phi \cap \Psi$.

	The \emph{flag graph} $\fgraph_{\calP}$ of an $n$-polytope $\calP$ is a simple, $n$-regular graph,
	with vertices corresponding to the flags of $\calP$, and with an edge labeled $i$ between two vertices
	whenever the corresponding flags are $i$-adjacent. The flag graph captures all of the information
	of a polytope; each $j$-face corresponds to a connected component of the graph obtained by
	deleting all edges labeled $j$, and two faces are incident if their corresponding components
	intersect. 
	
	\emph{Maniplexes} are generalizations of polytopes that are essentially connected graphs resembling
	flag graphs (see \cite{maniplexes}). In a maniplex, each flag has a well-defined $i$-adjacent flag for each $i$, but the diamond condition may still fail. Strong connectivity is also not required for maniplexes. Their main feature is that if $i,j \in \{0, \dots, n-1\}$ and $|i-j| \ge 2$ then every node of the flag graph (maniplex) belongs to an alternating square with edges labelled $i$ and $j$. Every $n$-polytope can be viewed as an $(n-1)$-maniplex. 
	
	Many operations 
	on polytopes produce structures which may not be polytopes, and so it is useful to have a broader context
	to work in. The reader who is used to working with \emph{pre-polytopes} (posets that satisfy the first
	three conditions for a polytope but not necessarily the fourth) can mentally change every
	`maniplex' to `pre-polytope' with little loss of accuracy. Strictly speaking, whenever we
	work with maniplexes and polytopes together, we should convert the polytope to its flag
	graph, but we will not bother with that formality.
	
	A maniplex is \emph{polytopal} if it is (the flag graph of) a polytope. We describe
	here one characterization of polytopality that will be helpful later, taken from \cite{poly-mani}.
	Let $\calM$ be an $n$-maniplex. The $i$-faces of $\calM$ are defined to be the connected
	components of $\calM_i$, the subgraph of $\calM$ obtained by deleting all $i$-edges.
	If $F$ is an $i$-face and $G$ is a $j$-face, then we say that $F < G$ if $i < j$ and $F \cap G \neq \emptyset$.
	A chain of $\calM$ is a sequence $F_1 < \cdots < F_k$. We say that $\calM$ has the
	\emph{component intersection property} (CIP) if, for every chain $F_1 < \cdots < F_k$ of $\calM$,
	the subgraph $\bigcap_{i=1}^k F_i$ is connected. By \cite[Thm. 4.5]{poly-mani}, a maniplex is
	polytopal if and only if it has the CIP.	
	
	A \emph{rooted $n$-polytope} is a pair $(\calP, \Phi)$, where $\calP$ is an $n$-polytope
	and $\Phi$ is a flag of $\calP$. Rooted maniplexes are defined analogously. We will refer to the polytope
	part of a rooted polytope as the \emph{underlying polytope}. We will sometimes refer to an underlying
	polytope even if, strictly speaking, the object we are working with might be a non-polytopal
	rooted maniplex.
	
\subsection{Automorphisms and $k$-orbit polytopes}	
		
	An \emph{isomorphism} of $n$-polytopes is a bijection that preserves rank and the partial order.
	The isomorphisms from $\calP$ to itself form the \emph{automorphism group} of $\calP$,
	denoted $\G(\calP)$. For rooted polytopes, we say that $(\calP, \Phi)$ is isomorphic to 
	$(\calQ, \Psi)$ if there is an isomorphism of polytopes $\varphi: \calP \to \calQ$ that sends $\Phi$ to $\Psi$.
	
	The group $\G(\calP)$ acts freely (semiregularly) on the flags of $\calP$. A \emph{$k$-orbit polytope}
	is one where $\G(\calP)$ has $k$ orbits on the flags. The one-orbit polytopes are called \emph{regular},
	and the two-orbit polytopes such that adjacent flags always lie in different orbits are called \emph{chiral}.
	A one-orbit maniplex is called \emph{reflexible}.
	
	Taking the quotient of $\fgraph_{\calP}$ by $\G(\calP)$ yields the \emph{symmetry type graph}
	$T(\calP)$ of $\calP$ \cite{stg}. In other words, $T(\calP)$ has one vertex for each
	flag orbit of $\calP$, and two vertices are connected by an edge labeled $i$ if, whenever
	$\Phi$ is in the orbit corresponding to one of the vertices, then $\Phi^i$ is in the orbit
	corresponding to the other vertex. If $\Phi$ and $\Phi^i$ are in the same orbit, then
	$T(\calP)$ has a semi-edge labeled $i$ at the corresponding vertex.
	Thus, if $\calP$ is a $k$-orbit $n$-polytope, then $T(\calP)$ is an $n$-regular graph on $k$ vertices.
	
	It is sometimes convenient to take the quotient of $\fgraph_{\calP}$ by a subgroup of
	$H$ of $\G(\calP)$, to get a representation of the flag orbits under the action of $H$.
	For example, if $\calP$ is a convex polytope (or an abstract polytope that is realized
	in euclidean space), then we can take the quotient of $\fgraph_{\calP}$ by the group
	$G(\calP)$ of geometric symmetries of $\calP$ to get the geometric symmetry type graph
	of $\calP$. 
	
	The automorphism group of a regular $n$-polytope $\calP$ has a standard form. If we fix a base flag
	$\Phi$, then for each $i \in \{0, \ldots, n-1\}$, there is an automorphism $\rho_i$ that
	sends $\Phi$ to $\Phi^i$. These automorphisms generate $\G(\calP)$ and satisfy (at least)
	the relations
	\begin{equation}
	\label{eq:involutions}
	\rho_i^2 = \eps \textrm{ (for all $i$)},
	\end{equation}
	\begin{equation}
	\label{eq:string}
	(\rho_i \rho_j)^2 = \eps \textrm{ (whenever $|i-j| > 1$)}.
	\end{equation}
	The automorphism group of a regular polytope also satisfies the following \emph{intersection condition}
	for all subsets $I$ and $J$ of $\{0, \ldots, n-1\}$:
	\begin{equation}
	\label{eq:int-cond}
	\langle \rho_i \mid i \in I \rangle \cap \langle \rho_i \mid i \in J \rangle = \langle \rho_i \mid i \in I \cap J \rangle.
	\end{equation}
	
	Any group $\G = \langle \rho_0, \ldots, \rho_{n-1} \rangle$ 
	that satisfies Equations~\ref{eq:involutions}, \ref{eq:string}, and \ref{eq:int-cond} 
	(with respect to the distinguished set of generators) is called a \emph{string C-group}. 
	More generally, a group that satisfies Equations~\ref{eq:involutions} and \ref{eq:string}
	is called a string group generated by involutions (sometimes abbreviated
	\emph{sggi}). Given a string C-group $\G$,
	we can build a regular polytope $\calP$ as a coset geometry of $\G$ such that
	$\G(\calP)$ is isomorphic to $\G$ (not just as abstract groups, but with the same
	generating set and relations).

	The automorphism groups of chiral polytopes satisfy equations that are analogous to 
	Equations~\ref{eq:involutions}, \ref{eq:string}, and \ref{eq:int-cond} \cite{chiral}. 
	As with regular polytopes, it is possible to build a chiral polytope as a coset geometry of such a group.
	To what extent can we do the same thing for
	general $k$-orbit polytopes? Using the symmetry type graph of $\calP$, it is possible
	to describe a standard generating set for $\G(\calP)$ (see \cite[Thm. 5.2]{stg}). Is
	it always possible to recover the structure of $\calP$ from such a group?
	
	\begin{problem}
	Given a distinguished generating set for the automorphism group of a $k$-orbit polytope, what analogue of the intersection
	condition holds?
	\end{problem}
	
	\begin{problem}
	\label{prop:coset-geo}
	Describe a way to build a general polytope as a coset geometry of its automorphism group,
	given a distinguished set of generators.
	\end{problem}
	
	See \cite{two-orbit} for work on these problems in the context of two-orbit polyhedra.

\subsection{The flag action}	
	
	Let 
	\[ W = [\infty, \ldots, \infty] = \langle r_0, \ldots, r_{n-1} \rangle, \]
	the universal Coxeter group with defining relations $r_i^2 = \eps$ for each $i$ and $(r_i r_j)^2 = \eps$ 
	whenever $|i - j| \geq 2$. There is a natural action of $W$ on the flags on any polytope,
	given by $r_i \Phi = \Phi^i$. Several important properties of this action can be found
	in \cite[Thm. 4.14]{mixing-and-monodromy}; we reiterate a few of them here.

	\begin{proposition}
	\label{prop:flag-action-properties}
	Let $\calP$ be an $n$-polytope, and let $\Phi$ be any flag of $\calP$.
	\begin{enumerate}
	\item The number of flags of $\calP$ is equal to $[W: \Stab_W(\Phi)]$.
	\item The number of flag orbits of $\calP$ is equal to $[W: \Norm_W(\Stab_W(\Phi))]$.
	\item $\G(\calP) \cong \Norm_W(\Stab_W(\Phi)) / \Stab_W(\Phi)$.
	\item Two flags $\Phi$ and $\Psi$ are in the same flag orbit if and only if $\Stab_W(\Phi) =
		\Stab_W(\Psi)$.
	\end{enumerate}
	\end{proposition}

	\pref{flag-action-properties}(d) implies that $(\calP, \Phi) \cong (\calQ, \Psi)$ if and only if 
	$\Stab_W(\Phi) = \Stab_W(\Psi)$. This makes it convenient to identify a rooted polytope $(\calP, \Phi)$ 
	with the subgroup $\Stab_W(\Phi)$. In fact, it is possible to reconstruct $\calP$ from $\Stab_W(\Phi)$. Let 
	${\mathcal U} = \{\infty, \ldots, \infty\}$ be the {\em universal} $n$-polytope, with $\G(\calU) = W$.
	Given a subgroup $S$ of $W$, the quotient ${\mathcal U}/S$ is a flag-connected
	poset of rank $n$ (see \cite[Prop. 2D3]{arp}). Equivalently, we can build a poset from $W/S$ using
	double cosets, as in \cite{all-polytopes-are-quotients}. In any case,
	if $S = \Stab_W(\Phi)$, then $\calU / S \cong \calP$. Furthermore, we may identify
	flags in $\calU / S$ with cosets of $S$ in $W$, and $(\calU/S, S) \cong (\calP, \Phi)$.
	
	\begin{definition}
	Let $(\calP, \Phi)$ be a rooted polytope, with $S = \Stab_W(\Phi)$.
	The \emph{canonical representation} of $(\calP, \Phi)$ is the rooted polytope
	\[ \calU(S) := (\calU/S, S). \]
	\end{definition}

	A $k$-orbit polytope $\calP$ has $k$ distinct flag-stabilizers under the action of $W$.
	The flags in any given orbit all have the same flag-stabilizer, and so
	the vertices of the symmetry type graph $T(\calP)$ can be identified with these
	stabilizers. In this context, two vertices of $T(\calP)$ are connected by an 
	$i$-edge if and only if the corresponding stabilizers are conjugate by $r_i$.
	In other words, $T(\calP)$ simply represents the action by conjugation of $W$ on the flag-stabilizers,
	and this can be analyzed purely abstractly as a permutation group. For example, this
	is essentially the tactic used in \cite{k-orbit-maps}.
	
	Each flag-stabilizer has an associated normalizer, and distinct flag-stabilizers may have distinct normalizers. The universal
	group $W$ permutes the normalizers in just the same way as it permutes the stabilizers.
	Using this fact, along with the fact that $\G(\calP) \cong \Norm_W(\Stab_W(\Phi)) / \Stab_W(\Phi)$,
	it is possible to find a small generating set for $\G(\calP)$ by using $T(\calP)$;
	see \cite[Theorem 5.2]{stg}. Essentially, any closed walk in $T(\calP)$ corresponds to an element
	of the normalizer of the starting orbit, and so finding generators for the normalizer
	amounts to finding walks which generate all closed walks at a given vertex.

	\begin{example}
	Suppose $\calP$ is a triangular prism. Then the symmetry type graph of $\calP$ (with
	semi-edges suppressed) is:
	\[ \xymatrix @*[o] @*=<0.4em> {
	\circ \ar@{-}[r]^{1} & \circ \ar@{-}[r]^{2} & \circ
	}. \]
	Consider flags $\Phi_1, \Phi_2,$ and $\Phi_3$, one for each orbit (from left to right). 
	Then:
	\begin{align*}
	\G(\calP, \Phi_1) &= \langle \alpha_0, \, \alpha_2, \, \alpha_{1,0,1}, \, \alpha_{1,2,1,2,1} \rangle, \\
	\G(\calP, \Phi_2) &= \langle \alpha_0, \, \alpha_{1,0,1}, \, \alpha_{1,2,1}, \, \alpha_{2,1,2} \rangle, \\
	\G(\calP, \Phi_3) &= \langle \alpha_0, \, \alpha_1, \, \alpha_{2,1,0,1,2}, \, \alpha_{2,1,2,1,2} \rangle,
	\end{align*}
	where $\alpha_{i_1, \ldots, i_m}$ is the automorphism that sends the base flag $\Phi_j$
	to $\Phi_j^{i_1 \cdots i_m}$. We see that different choices of base flag yield 
	different generators for $\G(\calP)$.
	\end{example}
	
	The action of $W$ on the flags of $\calP$ gives rise to the \emph{monodromy group} of $\calP$
	(also called the \emph{connection group} in \cite{parallel-product} and elsewhere),
	which describes a regular cover of $\calP$ (see Section~3.1).
	Formally, the flag action is a homomorphism $\pi: W \to \Sym({\mathscr F}(\calP))$,
	where ${\mathscr F}(\calP)$ is the set of flags of $\calP$, 
	and the monodromy group of $\calP$ (denoted $\Mon(\calP)$) is the image of $\pi$. Equivalently,
	\[ \Mon(\calP) \cong W / \ker \pi = W / \Core_W(S), \]
	where $S$ is the stabilizer of an arbitrary flag of $\calP$, and $\Core_W(S)$ is the intersection
	of all conjugates of $S$ \cite[Lem. 4.12]{mixing-and-monodromy}.
	Describing the monodromy group of several families of polytopes is an active area of research;
	see \cite{monodromy-pyramid, monodromy-truncated-simplex, prism-covers, archimedean2, archimedean1}.

\section{Coverings and mixing}
\subsection{Coverings}

	The most commonly used functions between polytopes are \emph{coverings}:
	surjective functions that preserve rank, incidence, and flag adjacency. Formally,
	a surjective function $\varphi: \calP \to \calQ$ is a covering if:
		\begin{enumerate}
		\item $F \leq_{\calP} G$ implies $F \varphi \leq_{\calQ} G \varphi$,
		\item $\rank F \varphi = \rank F$ for all faces $F$, and
		\item $\Phi^i \varphi = (\Phi \varphi)^i$ for all flags $\Phi$ and all $i \in \{0, \ldots, n-1\}$.
		\end{enumerate}
	If there is a covering $\varphi: \calP \to \calQ$, then
	we say that $\calP$ \emph{covers} $\calQ$ (and write $\calP \covers \calQ$).
	If there is a covering from $\calP$ to $\calQ$ that sends $\Phi$ to $\Psi$, then
	we say that $(\calP, \Phi)$ covers $(\calQ, \Psi)$ (and write $(\calP, \Phi)
	\covers (\calQ, \Psi)$). If $(\calP, \Phi) \covers (\calQ, \Psi)$, then
	$(\calP, \Phi^i) \covers (\calQ, \Psi^i)$ for every $i \in \{0, \ldots, n-1\}$.

	Coverings can also be defined in terms of flag graphs or flag stabilizers:
	
	\begin{proposition}
	If $(\calP, \Phi)$ and $(\calQ, \Psi)$ are rooted $n$-polytopes, then the following
	are equivalent:
	\begin{enumerate}
	\item $(\calP, \Phi) \covers (\calQ, \Psi)$.
	\item There is a surjective color-preserving graph homomorphism from $\fgraph_{\calP}$ to
	$\fgraph_{\calQ}$ that sends $\Phi$ to $\Psi$.
	\item $\Stab_W(\Phi) \leq \Stab_W(\Psi)$.
	\end{enumerate}
	\end{proposition}

	The definition of coverings extends naturally to pre-polytopes and maniplexes.
	When a regular pre-polytope $\calP$ covers a regular polytope $\calQ$, the
	\emph{quotient criterion} says that if $\calP$ and $\calQ$ have isomorphic facets, then
	$\calP$ is itself a polytope \cite[Thm. 2E17]{arp}.
	The same is true for chiral polytopes \cite[Lem. 3.2]{chiral-mix}. 
	Both results use algebraic arguments that seem difficult to generalize to arbitrary polytopes.
	Perhaps the Component Intersection Property (see the end of Section~\ref{sec:polytopes}) could
	be used to develop a combinatorial version of the quotient criterion that could be used
	for polytopes in general.
	
	\begin{problem}
	Find a combinatorial analogue of the quotient criterion.
	\end{problem}

	Every $n$-polytope is covered by the regular universal $n$-polytope $\calU$ with
	$\G(\calU) = W = [\infty, \ldots, \infty]$. Furthermore, every \emph{finite} polytope is covered by a 
	\emph{finite} regular polytope \cite{finite-covers}.
	How can we find the smallest regular polytope that covers a given polytope?

	If we drop the requirement that the regular cover must be polytopal, then every polytope $\calP$
	has a unique minimal regular cover $\calR$. That is, if $\calQ$ is a regular polytope
	(or reflexible maniplex) such that $\calQ \covers \calP$, then $\calQ \covers \calR$. We obtain $\calR$ by building a poset
	from the group $\Mon(\calP)$ such that $\G(\calR) \cong \Mon(\calP)$.
	If $\Mon(\calP)$ is a string C-group, then $\calR$ will be a polytope, and thus the minimal regular
	polytope that covers $\calP$ \cite[Prop. 3.16]{mixing-and-monodromy}. In rank 3, the monodromy
	group is always a string C-group, and so every polyhedron has a unique polytopal minimal
	regular cover \cite[Cor. 6.2]{mixing-and-monodromy}.

	In light of \cite[Prop. 3.16]{mixing-and-monodromy}, it is useful to know when $\Mon(\calP)$ is a
	string C-group. See \cite[Section 6]{mixing-and-monodromy} for some partial results.
	
	\begin{problem}
	Give a characterization of the polytopes $\calP$ such that $\Mon(\calP)$ is a string C-group.
	\end{problem}
	
	If $\Mon(\calP)$ is not a string C-group, it seems it may still be the case that $\calP$ has a unique
	minimal regular polytope that covers it. No examples of such a polytope are known.
	
	\begin{problem}
	Find a polytope $\calP$ such that $\Mon(\calP)$ is not a string C-group, but $\calP$ does have
	a unique minimal regular polytopal cover; or show the non-existence of such a polytope.
	\end{problem}

	If polytopes with a unique minimal regular cover represent one extreme, then the \emph{Tomotope}
	${\mathcal T}$ represents the other extreme. The Tomotope is a 4-polytope such that $\Mon({\mathcal T})$ is not a
	string C-group \cite{tomotope}. Furthermore, 
	% there is no regular polytope $\calR$ such that any regular polytope that covers ${\mathcal T}$ also covers $\calR$. In fact, 
	there are infinitely many finite regular polytopes that cover ${\mathcal T}$, with no covering
	relations between them.
	
	One can also pose the problem of minimal regular covers completely in terms of groups.
	Essentially, one wants to find a minimal string C-group that covers a given string group generated
	by involutions.
	
	\begin{problem}
	Given a string group generated by involutions, determine the minimal string C-groups
	that cover it.
	\end{problem}

	The \emph{chirality group} of a chiral polytope $\calP$ is a subgroup $X(\calP)$ of
	$\G(\calP)$ that gives some information about how far $\calP$ is from being regular.
	In particular, 
	\[ |\Mon(\calP)| = 2|\G(\calP)| \cdot |X(\calP)|. \]
	In other words, $|X(\calP)|$ is the index of the automorphism group of $\calP$ on the automorphism group of the smallest regular maniplex covering $\calP$. This number equals the index of the automorphism group of the largest regular structure (it may even fail to be a maniplex) covered by $\calP$ on the automorphism group of $\calP$.
	
	If $S_1$ and $S_2$ are the
	two distinct flag-stabilizers of $\calP$, then
	\[ X(\calP) = S_1 S_2 / S_1 \cong S_2 / (S_1 \cap S_2). \]
	The same definition works for any two-orbit polytopes. It would be useful to extend this definition to an \emph{irregularity group} of 
	a $k$-orbit polytope with $k \geq 3$. This is not entirely straightforward. 
	Unlike with two-orbit polytopes, in the general case, different flag-stabilizers may
	have different normalizers. This means that given flag-stabilizers $S_1$ and $S_2$,
	expressions like $S_2 / (S_1 \cap S_2)$ may not
	represent groups, as there is no guarantee that $S_1 \cap S_2$ is normal in $S_2$.
	Furthermore, with more flag-stabilizers, there are many more such expressions which
	(in general) provide groups or coset spaces of different sizes.
	
	\begin{problem}
	Determine the ``correct'' way to generalize the chirality group to $k$-orbit polytopes.
	\end{problem}
	
	Even when the structure of $\Mon(\calP)$ may be difficult to determine, it would be useful to
	find its size. As with chiral polytopes, it may be that knowing the size of the (generalization of)
	the chirality group would help with this problem.
	
	\begin{problem}
	Given a polytope $\calP$, determine $|\Mon(\calP)|$.
	\end{problem}

	Let us find some bounds on $|\Mon(\calP)|$ that use only basic information about $\calP$.
	Suppose $\calP$ is a $k$-orbit polytope. Let $S_1, \ldots, S_k$ be the $k$ distinct flag-stabilizers 
	of $\calP$, and let $N_1, \ldots, N_k$ be the corresponding normalizers in $W$. Then $W$
	permutes the normalizers by conjugation, and it follows that 
	\[ [W: N_1 \cap \cdots \cap N_k] \leq k!. \]
	Now, since each $S_i$ is a subgroup of $N_i$, a standard result in group theory gives
	\[ [N_1 \cap \cdots \cap N_k : S_1 \cap \cdots \cap S_k] \leq [N_1: S_1] \cdots [N_k: S_k]. \]
	Since $\G(\calP) \cong N_i / S_i$ for each $i$, the right hand side is just
	$|\G(\calP)|^k$. Putting everything together with the fact that $\Mon(\calP) \cong
	W / (S_1 \cap \cdots \cap S_k)$ gives us
	\[ |\Mon(\calP)| \leq [W: N] \cdot |\G(\calP)|^k \leq k! \, |\G(\calP)|^k. \]
	
	To get a lower bound, let us define $T_1, \ldots, T_k$ by $T_i = \cap_{j \neq i} S_j$.
	So $T_i$ fixes all flags in $k-1$ of the orbits, while permuting the flags in the last
	orbit. Then since $T_i$ fixes the orbit of every flag, it follows that $T_i \subseteq N_1 \cap
	\cdots \cap N_k$. In fact, $T_i$ is a normal subgroup of the latter, since $T_i$ is an intersection
	of subgroups $S_j$ each normalized by $N_j$.

	Let $S = S_1 \cap \cdots \cap S_k$ and $N = N_1 \cap \cdots \cap N_k$. 
	The group $S$ is normal in each $T_i$ since $T_i$ is contained in each normalizer $N_j$. 
	So $S$ is normal in $T_1 \cdots T_k$, which is itself normal in $N$. It follows that
	$(T_1 \cdots T_k) / S$ is normal in $N/S$. We can bound the size of the former as follows.
	Note that 
	\[ (T_1 \cdots T_{i-1} T_{i+1} \cdots T_k) \cap T_i \subseteq S_i \cap T_i = S. \]
	Since $S$ is contained in each $T_i$, it follows that $(T_1 \cdots T_{i-1} T_{i+1} \cdots T_k) \cap T_i = S$ for each $i$.
	Then
	\[ ((T_1/S) \cdots (T_{i-1}/S) (T_{i+1}/S) \cdots (T_k/S)) \cap (T_i/S) = \langle \eps \rangle. \]
	Then by a standard result in group theory (see \cite[Cor. 8.7]{hungerford}), 
	\[ (T_1 / S) \cdots (T_k / S) \cong (T_1 / S) \times \cdots \times (T_k / S). \]
	Since the groups $T_i/S$ are all conjugate in $W/S$, it follows that the above has order
	$|T_1/S|^k$. Thus
	\[ |\Mon(\calP)| = [W: N] \cdot |N/S| \geq [W: N] \cdot |T_1/S|^k. \]
	We summarize in the following proposition.
	
	\begin{proposition}
	Let $\calP$ be a $k$-orbit polytope. Let $S$ be the intersection of all $k$ distinct flag-stabilizers of $\calP$, let
	$N$ be the intersection of the normalizers of every flag-stabilizer, and let $T$ be the intersection of any
	$k-1$ distinct flag-stabilizers. Then
	\[ [W: N] \cdot |T/S|^k \leq |\Mon(\calP)| \leq [W: N] \cdot |\G(\calP)|^k \leq k! \, |\G(\calP)|^k. \]
	\end{proposition}

\subsection{Mixing}	

	The \emph{mix} of two regular or chiral polytopes (defined in \cite{mix-face, chiral-mix}) constructs their minimal common cover. 
	This is helpful for constructing chiral polytopes, as well as polytopes that are invariant under certain operations; 
	see \cite{chiral-mix, self-dual-chiral, var-gps, mixing-and-monodromy}. There is a natural candidate for
	the minimal common cover of two rooted polytopes:
	
	\begin{definition}
	The \emph{mix} of $(\calP, \Phi)$ and $(\calQ, \Psi)$, denoted $(\calP, \Phi) \mix (\calQ, \Psi)$,
	is the rooted maniplex $\calU(\Stab_W(\Phi) \cap \Stab_W(\Psi))$.
	\end{definition}

	The rooted maniplex $(\calP, \Phi) \mix (\calQ, \Psi)$ satisfies the following universal property:
	every rooted maniplex that covers $(\calP, \Phi)$ and $(\calQ, \Psi)$
	also covers $(\calP, \Phi) \mix (\calQ, \Psi)$. This follows immediately from the fact
	that $(\calR, \Lambda)$ covers $(\calP, \Phi)$ if and only if $\Stab_W(\Lambda) \leq \Stab_W(\Phi)$.

	It is of course possible to mix multiple rooted polytopes together, and this operation is naturally
	commutative and associative (as it must be, since it simply corresponds to taking the intersection
	of flag-stabilizers). 
	
	\begin{example}\label{ex:mixFlagOrbits}
	Let $\calP$ be a $k$-orbit polytope, and let $\Phi_1, \ldots, \Phi_k$ be
	flags of $\calP$, one from each flag orbit. Then 
	\[ (\calP, \Phi_1) \mix \cdots \mix (\calP, \Phi_k) \]
	is the minimal regular cover of $\calP$, with automorphism group isomorphic to $\Mon(\calP)$.
	\end{example}

	Our definition of the mix bears little resemblance to the usual definition used for mixing
	regular polytopes, which describes a ``diagonal subgroup'' of the direct product of the
	automorphism groups (see \cite[Def 5.1]{mixing-and-monodromy}). There are two good reasons
	for this. The first is that the automorphism groups of two arbitrary polytopes will usually
	have different generating sets, so that there is no notion of a diagonal subgroup of the
	direct product. The second is that even when we can overcome the first difficulty, we end
	up with a group, and it is not immediately clear how to build the correct polytope with this
	automorphism group (see Problem~\ref{prop:coset-geo}).

	As an aside, we note that the mix of chiral polytopes (as used in \cite{chiral-mix, mix-ch} and elsewhere)
	was already using rooted polytopes, just with a different (and perhaps slightly misleading) language.
	
	We can also define the mix of rooted polytopes using their flag graphs.
	Consider the flag graphs $(\fgraph_{\calP}, \Phi)$ and $(\fgraph_{\calQ}, \Psi)$. 
	We define a new graph $\fgraph$ with vertex set $V(\fgraph_{\calP}) \times V(\fgraph_{\calQ})$ and with 
	an $i$-edge between $(\Phi_1, \Psi_1)$ and $(\Phi_2, \Psi_2)$ if and only if there is an $i$-edge between 
	$\Phi_1$ and $\Phi_2$ in $\fgraph_{\calP}$ and an $i$-edge between $\Psi_1$ and $\Psi_2$ in $\fgraph_{\calQ}$. 
	We define the mix (or \emph{parallel product}) of $\fgraph_{\calP}$ with $\fgraph_{\calQ}$ (denoted 
	$\fgraph_{\calP} \mix \fgraph_{\calQ}$) to be $\fgraph$,
	and we define the mix of $(\fgraph_{\calP}, \Phi)$ with $(\fgraph_{\calQ}, \Psi)$ 
	(denoted $(\fgraph_{\calP}, \Phi) \mix (\fgraph_{\calQ}, \Psi)$) to be the
	connected component of $\fgraph$ that contains $(\Phi, \Psi)$.
	It is easy to show from the universal property of the mix that $(\fgraph_{\calP}, \Phi) \mix (\fgraph_{\calQ}, \Psi)$
	is the flag graph of $(\calP, \Phi) \mix (\calQ, \Psi)$. This alternative definition of the mix is a generalization of the definition of parallel product of maps in \cite{parallel-product}.
	
	The mix of $T(\calP)$ with $T(\calQ)$, or of $T(\calP, \Phi)$ with $T(\calQ, \Psi)$, can be defined in the same way.
	There is a nice relationship between the mix of symmetry type graphs of rooted polytopes, and
	the symmetry type graph of the mix.
	
	\begin{proposition}
	\label{prop:sym-type-graph}
	If $(\calP, \Phi)$ and $(\calQ, \Psi)$ are rooted $n$-polytopes, then there is a surjective
	color-preserving graph homomorphism from $T(\calP, \Phi) \mix T(\calQ, \Psi)$ to
	$T \big( (\calP, \Phi) \mix (\calQ, \Psi) \big)$. 
	\end{proposition}
	
	\begin{proof}
	Let $S = \Stab_W(\Phi)$ and $S' = \Stab_W(\Psi)$. We may identify the vertices of 
	$T(\calP, \Phi) \mix T(\calQ, \Psi)$ with pairs $(w^{-1} S w, w^{-1} S' w)$, where $w \in W$.
	Similarly, we may identify the vertices of $T\left( (\calP, \Phi) \mix (\calQ, \Psi) \right)$
	with $w^{-1} (S \cap S') w$. Then the map that sends each $(w^{-1} S w, w^{-1} S' w)$ to
	$w^{-1} (S \cap S') w$ has the desired properties.
	\end{proof}
	
In general the graph homomorphism in Proposition \ref{prop:sym-type-graph} is not an isomorphism. The simplest example is to consider $\calP = \calQ$ a two-orbit polytope with $\Phi$ and $\Psi$ in distinct flag orbits (see Example \ref{ex:mixFlagOrbits}).
	
	\begin{corollary}
	If $\calP$ is a $k$-orbit polytope and $\calQ$ is an $m$-orbit polytope,
	then the underlying polytope of $(\calP, \Phi) \mix (\calQ, \Psi)$ has at most $km$ flag orbits for any choice of $\Phi$ and $\Psi$.
	\end{corollary}

	\begin{corollary}
	Let $(\calP, \Phi)$ and $(\calQ, \Psi)$ be rooted $k$-orbit polytopes with the same symmetry type graph
	$G$. If the orbits of $\Phi$ and $\Psi$ correspond to the same vertex of $G$, then $(\calP, \Phi) \mix
	(\calQ, \Psi)$ has at most $k$ orbits.
	\end{corollary}

	The mix of two rooted polytopes is not always polytopal. Several results in the literature
	describe sufficient conditions for the mix of regular or chiral polytopes to be polytopal; see
	\cite[Thm. 3.7]{mix-ch}, \cite[Thm. 5.12]{mixing-and-monodromy}, \cite[Prop. 5.15]{mixing-and-monodromy}.
	It would be useful to find similar results for arbitrary polytopes.
	
	\begin{problem}
	Give a characterization of pairs of rooted polytopes whose mix is itself a polytope.
	\end{problem}
	
	Here is one such result which generalizes \cite[Thm. 5.15]{mixing-and-monodromy}.
	
	\begin{theorem}
	The mix of two rooted polyhedra is polytopal.
	\end{theorem}
	
	\begin{proof}
	Let $(\calP, \Phi)$ and $(\calQ, \Psi)$ be rooted polyhedra, and
	let $\fgraph = (\fgraph_{\calP}, \Phi) \mix (\fgraph_{\calQ}, \Psi)$. 
	To show that $(\calP, \Phi) \mix (\calQ, \Psi)$ is polytopal, it suffices to show
	that $\fgraph$ satisfies the CIP (see the end of Section~\ref{sec:polytopes} and \cite[Thm. 4.5]{poly-mani}). 
	For $i \in \{0, 1, 2\}$, let $F_i$ be an $i$-face of $\fgraph$. (Recall that this means
	that $F_i$ is a connected component of $\fgraph_i$, where the latter is obtained from
	$\fgraph$ by deleting all $i$-edges.) In this case, the CIP consists of three nontrivial
	conditions: we need to show that $F_i \cap F_j$ is either empty or connected for
	$(i, j) \in \{(0, 1), (1, 2), (0, 2)\}$.

	Suppose to the contrary that $F_0 \cap F_1$ is nonempty and disconnected. The face $F_1$ must consist of
	a single 4-cycle with two $0$-edges and two $2$-edges. Then the only way for $F_0 \cap F_1$
	to be disconnected is for it to consist of both $2$-edges of $F_1$. Then $\fgraph$
	has a subgraph with the following form:
	\[ \xymatrix {
	(\Phi_1, \Psi_1) \ar@{-}[d]_2 \ar @{-} @/^/ [rr]^0 \ar@{.}@/_/[rr]_{121\cdots21} & & (\Phi_4, \Psi_4) \ar@{-}[d]^2 \\
	(\Phi_2, \Psi_2) \ar@/_/@{-}[rr]_0 \ar@{.}@/^/[rr]^{121\cdots21} & & (\Phi_3, \Psi_3)
	} \]
	There is a covering from $\fgraph$ to $\fgraph_{\calP}$, obtained by simply keeping
	$\Phi_i$ from each pair $(\Phi_i, \Psi_i)$. We note that the flags $\Phi_1, \ldots, \Phi_4$
	are incident on a single edge of $\calP$, and since $\calP$ is a polyhedron, it follows that
	they are distinct. So we get essentially the same subgraph in $\fgraph_{\calP}$. But this then
	gives us a disconnected intersection in $\calP$, violating its polytopality. So $F_0
	\cap F_1$ cannot be disconnected. A dual argument shows that the same is true of $F_1 \cap F_2$. 
	
	Now suppose that $F_0 \cap F_2$ is nonempty and disconnected. Then $F_0$ and $F_2$ share at least two
	$1$-edges. Therefore, the subgraph induced by the vertices of $F_2$ has either the form:
	\[ \xymatrix {
	(\Phi_1, \Psi_1) \ar@{-}[d]_1 \ar@{.}@/^/[rr]^{010\cdots10} \ar@{.}@/_/[rr]_{212\cdots12} & & (\Phi_4, \Psi_4) \ar@{-}[d]^1 \\
	(\Phi_2, \Psi_2) \ar@{.}@/_/[rr]_{010\cdots10} \ar@{.}@/^/[rr]^{212\cdots12} & & (\Phi_3, \Psi_3)
	} \]
	or
	\[ \xymatrix {
	(\Phi_1, \Psi_1) \ar@{-}[dd]_1 \ar@{.}@/^/[rr]^{010\cdots10} \ar@{.}[dr] & & (\Phi_4, \Psi_4) \ar@{-}[dd]^1 \ar@{.}[dl]\\
	\ar @{} [rr]|{212\cdots12} & & \\
	(\Phi_2, \Psi_2) \ar@{.}@/_/[rr]_{010\cdots10} \ar@{.}[ur] & & \ar@{.}[ul] (\Phi_3, \Psi_3)
	} \]
	As before, this subgraph covers a corresponding subgraph in $\fgraph_{\calP}$ and one in
	$\fgraph_{\calQ}$. If every $\Phi_i$ is distinct, then we get a disconnected intersection
	in $\fgraph_{\calP}$, and if every $\Psi_i$ is distinct, then we get a disconnected intersection in
	$\fgraph_{\calQ}$. Since $\calP$ and $\calQ$ are both polyhedra, this cannot happen, 
	and so there must be some identification of flags $\Phi_i$ and flags $\Psi_i$. 
	Now, any closed walk along edges labeled 0
	and 1 (or along edges labeled 1 and 2) must have even length (in the flag graph of a polyhedron), 
	so we cannot identify flags that are separated by an odd walk. This rules out the second subgraph
	above. The only possible identification in the first subgraph is 
	$\Phi_1 = \Phi_3$, which also forces $\Phi_2 = \Phi_4$.
	For the same reason, we must have $\Psi_1 = \Psi_3$. But then $(\Phi_1, \Psi_1) = (\Phi_3, \Psi_3)$, 
	and those vertices were supposed to be distinct. So no identification is possible, and
	$F_0 \cap F_2$ cannot be disconnected.
	\end{proof}

	Can we define the mix of polytopes that are not rooted? Perhaps the most natural definition is
	via the following universal property:
	
	\begin{definition}
	Let $\calP$ and $\calQ$ be $n$-polytopes. Suppose that every polytope that covers $\calP$ and $\calQ$
	also covers some maniplex $\calR$. Then $\calR$ is the \emph{mix of $\calP$ with $\calQ$}, denoted
	$\calR = \calP \mix \calQ$.
	\end{definition}

	The mix of polytopes, unlike the mix of \emph{rooted} polytopes, is not always well-defined. 
	We start with a simple result.

	\begin{proposition}
	\label{prop:rooted-mix}
	If $\calP \mix \calQ$ is well-defined, then it is the underlying polytope of $(\calP, \Phi) \mix
	(\calQ, \Psi)$ for some choice of flags $\Phi$ and $\Psi$.
	\end{proposition}
	
	\begin{proof}
	Let $\calR = \calP \mix \calQ$. Since $\calR$ covers $\calP$ and $\calQ$, then it must cover
	the underlying polytope of $(\calP, \Phi) \mix (\calQ, \Psi)$ for some choice of $\Phi$
	and $\Psi$, both images of the same flag of $\calR$ under the quotients. Conversely, the underlying polytope of $(\calP, \Phi) \mix (\calQ, \Psi)$
	covers both $\calP$ and $\calQ$, and thus it covers $\calR$.
	\end{proof}

	Suppose that $\calP$ and $\calQ$ are chiral polytopes such that neither
	covers the other. By \pref{rooted-mix}, the mix of $\calP$ with $\calQ$ could only
	be the underlying polytope of $(\calP, \Phi) \mix (\calQ, \Psi)$ or $(\calP, \Phi^0) \mix
	(\calQ, \Psi)$. (The other two essentially different choices of flags yield the same two
	polytopes as these two choices.)
	In general, however, these two polytopes will be incomparable (i.e., neither
	will cover the other), and the mix of $\calP$ with $\calQ$ will be undefined.
	
	For example, consider the torus maps $\{4, 4\}_{(b, c)}$, which can be thought of as
	the quotient of the plane tiling by squares by the translation subgroup generated by
	translations by $(b, c)$ and $(-c, b)$ (see \cite[Sec. 1D]{arp}). These are chiral
	when $bc(b-c) \neq 0$, and in that case should be more properly thought of as rooted
	polytopes. Let $\{4, 4\}_{[b, c]}$ denote the underlying polytope of $\{4, 4\}_{(b,c)}$.
	If $\calP = \{4, 4\}_{[1,2]}$ and $\calQ = \{4, 4\}_{[1,4]}$, then
	both $\{4, 4\}_{[6,7]}$ and $\{4, 4\}_{[2,9]}$ are minimal covers of $\calP$ and
	$\calQ$, neither of which covers the other.

	Similar problems happen when we mix any two polytopes with the same symmetry type graph.
	On the other hand, if two polytopes have symmetry type graphs that are sufficiently different,
	then we are guaranteed a well-defined mix.

	\begin{proposition}
	\label{prop:connected-mix}
	Let $\calP$ and $\calQ$ be $n$-polytopes. Let $\Phi$ and $\Phi'$ be flags of $\calP$ in flag
	orbits ${\mathcal O}_1$ and ${\mathcal O}_1'$, and let
	$\Psi$ and $\Psi'$ be flags of $\calQ$ in flag orbits ${\mathcal O}_2$ and ${\mathcal O}_2'$. 
	Then there is a path from $({\mathcal O}_1, {\mathcal O}_2)$ to $({\mathcal O}_1', {\mathcal O}_2')$
	in $T(\calP) \mix T(\calQ)$ if and only if 
	there is an element $w \in W$ such that $w^{-1} \Stab_W(\Phi) w = \Stab_W(\Phi')$ and
	$w^{-1} \Stab_W(\Psi) w = \Stab_W(\Psi')$.
	\end{proposition}
	
	\begin{proof}
	This follows easily from the fact that an $i$-edge connects two vertices of $T(\calP)$ if
	and only if the corresponding flag-stabilizers are conjugate by $r_i$.
	\end{proof}
	
	\begin{proposition}
	\label{prop:well-defd-mix}
	Let $\calP$ and $\calQ$ be $n$-polytopes. If $T(\calP) \mix T(\calQ)$ is connected, then
	$\calP \mix \calQ$ is well-defined.
	\end{proposition}
	
	\begin{proof}
	Consider flags $\Phi$ and $\Phi'$ of $\calP$ and flags $\Psi$ and $\Psi'$ of $\calQ$. 
	Since $T(\calP) \mix T(\calQ)$ is connected, \pref{connected-mix} implies that
	$\Stab_W(\Phi) \cap \Stab_W(\Psi)$ is conjugate to $\Stab_W(\Phi') \cap \Stab_W(\Psi')$.
	It follows that the underlying polytope of $(\calP, \Phi) \mix (\calQ, \Psi)$ is the same as
	the underlying polytope of $(\calP, \Phi') \mix (\calQ, \Psi')$. 
	Since the flags were arbitrary, we see that all of the rooted mixes have the same
	underlying polytope $\calR$. Therefore, every polytope that covers $\calP$ and $\calQ$
	also covers $\calR$, and so $\calR = \calP \mix \calQ$.
	\end{proof}

	\begin{corollary}
	If $\calP$ and $\calQ$ are $n$-polytopes and $\calP$ is regular, then $\calP \mix \calQ$
	is well-defined.
	\end{corollary}

\section{Constructions of $k$-orbit polytopes}
\subsection{Basic existence questions}
	
	There has been very little systematic study of $k$-orbit polytopes with $k \geq 3$. Many basic
	questions remain unanswered. Perhaps the most basic question is whether there are $k$-orbit $n$-polytopes
	for every $k \geq 1$ and $n \geq 3$. If we allow polytopes with digonal sections,
	then the answer is yes \cite[Theorem 5.2]{Illanit}. However, such polytopes are often
	considered degenerate. If we forbid such polytopes, then the problem remains open. 
	
	\begin{problem}\label{p:OrbitsRank}
	Are there $k$-orbit $n$-polytopes for every $k \ge 1$ and every $n \ge 3$ such that no section of rank $2$ is a digon?
	\end{problem}

	The answer seems likely to be yes. Here is a construction that may serve as a foundation for answering
	this question.
	
	\begin{proposition}
	\label{prop:cube-stack}
	For every odd $k \geq 5$, there is a $k$-orbit polyhedron such that no section of rank $2$ is a digon.
	\end{proposition}
	
	\begin{proof}
	Consider $m \geq 2$ identical cubes in a single stack. Keep the edges and vertices that form the seams where
	one cube meets another, so that we get a polyhedron with $4m+4$ vertices, $8m+4$ edges, and
	$4m+2$ faces. (We could even change this to be a convex polyhedron by dilating the seams
	by different factors, effectively replacing the cubes with truncated pyramids.) It is clear that
	the flags that contain the top face and those that contain the bottom face are in the same
	orbit, and that no other flag is in that orbit. That gives us $16$ flags in that orbit, and since
	there are $8m+4$ edges, there are $32m+16$ flags total. This gives us $2m+1$ flag orbits.
	Finally, no face or vertex-figure is a digon.
	\end{proof}

	There are several minor modifications to \pref{cube-stack} that could be further applied to
	Problem~\ref{p:OrbitsRank}. For example, we could stack $n$-cubes, or cap off
	the stack with a pyramid at one or both ends. It seems that this alone will not be enough
	to settle the question in higher ranks, but perhaps just a little further expansion is needed.

	One way to specialize Problem~\ref{p:OrbitsRank} is to preassign the automorphism group.
	A group $\Gamma$ is the automorphism group of a regular $n$-polytope if and only if $\Gamma$ is a
	string C-group of rank $n$ (with respect to some set of generators). More generally, $\Gamma$ is the
	automorphism group of a reflexible maniplex if and only if $\Gamma$ is a string group generated by
	involutions. Finding a generalization of these principles to $k$-orbit
	polytopes would be useful, but seems quite difficult. 

	\begin{problem}
	Given a group $\Gamma$, what are the numbers $k \ge 1$ and $n \ge 3$ such that $\Gamma$ is the automorphism group of a $k$-orbit polytope or maniplex of rank $n$?
	\end{problem}

	Two nice partial results are already known. By \cite[Theorem 5.2]{Illanit}, if $\Gamma$ is a string C-group of rank $n$, then it is the automorphism group of a $k$-orbit maniplex of rank $n$ for every $k$. In fact, if $\Gamma$ satisfies
	some mild conditions, then it is actually the automorphism group of a polytope. In a different direction,
	it was proved in \cite{PreasGroup} that every group is the automorphism group of some abstract polytope. To the authors' knowledge, not much more is known.

	Specializing Problem \ref{p:OrbitsRank} in a different way, we can ask whether there are any maniplexes
	with a given symmetry type graph. There are a couple of obvious necessary conditions for a properly edge-colored $n$-regular graph with colors in $\{0,\dots, n-1\}$ to be the
	symmetry type graph of a maniplex: it must be connected, and for every $i$ and $j$ such that $|i - j| > 1$,
	walking along edges colored $i, j, i, j$ always gets us back to where we started. We shall call a graph {\em allowable} if it satisfies these properties. No other general restrictions
	on symmetry type graphs are known.
	
	\begin{problem}\label{p:allowableTypeGraph}
	Is every allowable graph the symmetry type graph of some maniplex?
	\end{problem}

	The answer to the previous problem is still not known even for two-orbit maniplexes (that is, for graphs with only two vertices). A few partial results are known, such as Theorem \ref{t:2ToTheP} below, obtained thanks to the construction $2^{\mathcal{P}}$.

	The polytope $2^{\mathcal{P}}$ was originally constructed by Danzer in \cite{Danzer1}. Given a {\em vertex-describable} $n$-polytope $\mathcal{P}$, (a polytope with the property that no two faces have the same sets of vertices), the polytope $2^{\mathcal{P}}$ is an $(n+1)$-polytope whose vertex-figures are all isomorphic to $\mathcal{P}$. If $\mathcal{P}$ has $m$ vertices then $2^{\mathcal{P}}$ has $2^m$ vertices, justifying the name. Danzer's construction is purely combinatorial, and consists of defining the $2^m$ vertices of $2^{\mathcal{P}}$ and then defining the remaining faces by determining their vertex sets with the partial order given by inclusion. In particular, $2^{\mathcal{P}}$ is also vertex-describable. The drawback of this construction is that it requires $\mathcal{P}$ to be not only a polytope, but also vertex-describable. 

	Later group-theoretical constructions of the polytope $2^{\mathcal{P}}$ were given in \cite[Chapter 8]{arp} through a twisting operation, and in \cite{pellicer-extension} using permutation groups. In both cases $\mathcal{P}$ no longer needs to be vertex-describable, but it is required to be regular.

	Finally, in \cite{products}, the dual of the following alternative definition of $2^{\mathcal{P}}$ is given. Let $\mathcal{P}$ be an $n$-maniplex with set of flags $\Omega$ and with vertices labeled $1, \dots, m$. We define a maniplex $2^{\calP}$ with vertex-set $\mathbb{Z}_2^m$ and flag set $\Omega \times \mathbb{Z}_2^m$, where
	\[(\Phi,x)^i = \left\{\begin{array}{ll}
	(\Phi^i,x) & \mbox{if $i\ge 1$},\\
	(\Phi,x^{j}) & \mbox{if $i = 0$},
	\end{array} \right.\]
	where $j$ is the vertex that $\Phi$ contains, and where $x^j$ differs from $x$ precisely in the $j^{th}$ entry. This definition does not require $\mathcal{P}$ to be vertex-describable, a polytope, or symmetric in any sense. Furthermore, it is easy to show that if $\mathcal{P}$ is a polytope then so is $2^{\mathcal{P}}$.

	For every $y \in \mathbb{Z}_2^m$, the maniplex $2^{\mathcal{P}}$ has symmetries $\tau_y$ that map $(\Phi,x)$ to $(\Phi,x+y)$. In particular, $(\Phi, x)$ and its $0$-adjacent flag $(\Phi,x^j)$ always belong to the same orbit. This shows that the symmetry type graph of $2^{\mathcal{P}}$ has semi-edges labelled $0$ at every node and that $2^{\mathcal{P}}$ is vertex-transitive.

	Given an automorphism $\gamma$ of $\mathcal{P}$ and a vertex $x = (x_1, \dots, x_m)$ of $2^{\mathcal{P}}$, we denote by $x\gamma$ the vector $(x_{(1)\gamma^{-1}}, \dots, x_{(m)\gamma^{-1}})$. Then the function $\hat{\gamma}$ mapping $(\Phi,x)$ to $(\Phi\gamma,x\gamma)$ induces an automorphism of $2^{\mathcal{P}}$. This shows that the stabilizer of the base vertex $(0,\dots,0)$ contains the automorphism group of $\mathcal{P}$ as a subgroup. Since the vertex-figure at $(0,\dots,0)$ is isomorphic to $\mathcal{P}$, the stabilizer of this vertex must in fact be isomorphic to the full automorphism group of $\mathcal{P}$. Then since $2^{\mathcal{P}}$ is vertex-transitive, it follows that $\mathcal{P}$ and $2^{\mathcal{P}}$ have the same number of flag orbits. Moreover, two flag orbits of $\mathcal{P}$ (say, the ones containing $\Phi$ and $\Phi^i$) are connected by an $i$-edge
	in $T(\calP)$ if and only if the corresponding flag orbits of $2^{\mathcal{P}}$ (that is, the ones containing $(\Phi,(0,\dots,0))$ and $(\Phi^i,(0,\dots,0))$) are connected by an $(i+1)$-edge in $T(2^{\calP})$.

	The previous analysis shows that the symmetry type graph of $2^{\mathcal{P}}$ can be obtained from that of $\mathcal{P}$ by increasing every label by $1$ and adding a semi-edge with label $0$ at each vertex. By considering $(2^{\mathcal{P}^{*}})^{*}$ we obtain the following theorem. (Recall that $\mathcal{P}^{*}$ denotes the dual of $\mathcal{P}$.) 

	\begin{theorem}\label{t:2ToTheP}
	Let $\fgraph$ be an allowable symmetry type graph of rank $n$ and assume that there is an $(n-1)$-maniplex (resp. $n$-polytope) whose symmetry type graph is $\fgraph$. Then there is an $n$-maniplex (resp. $(n+1)$-polytope) whose symmetry type graph is the graph $\hat{\fgraph}$ constructed from $\fgraph$ by attaching semi-edges labelled $n$ to all vertices.
	\end{theorem}

	If we allow degenerate polytopes (with digonal sections), we can obtain the same result in a simpler way. If $\calQ$ is the trivial
	extension of an $n$-polytope $\calP$ (consisting of two copies of $\calP$ ``glued back-to-back''; see
	Section 2.1), then we can obtain
	$T(\calQ)$ from $T(\calP)$ by attaching semi-edges labeled $n$ to each vertex. If we look at $T(\calQ^{*})$
	instead, then we get $T(\calP)$, with every label increased by 1, and with semi-edges labeled 0 at each vertex ---
	the same as we saw with $2^{\calP}$.
	
	By iterating the process of repeatedly taking either the construction $2^{\calP}$ or the trivial
	extension and dualizing, we can shift the labels of $T(\calP)$ by any desired amount. Similarly, by taking the dual of $2^{\calP*}$ or the trivial extension we can construct an $(n+1)$-maniplex from an $n$-maniplex $\calP$, whose symmetry type graph is that of $\calP$ with the addition of semi-edges with label $n$ at every vertex. Iterating this process we can construct maniplexes whose symmetry type graph is obtained by adding semi-edges at every vertex with labels greater than $n$ to the symmetry type graph of an $n$-maniplex.
	
	\begin{theorem} \label{thm:3-orbit}
	Every allowable symmetry type graph for three-orbit polytopes occurs as the symmetry type graph
	for a polytope.
	\end{theorem}
	
	\begin{proof}
	Proposition 4.1 in \cite{stg} shows the possible symmetry type graphs for three-orbit polytopes. We note that 
	every such graph can be seen as a `shifted' version of one of the allowable graphs for three-orbit polytopes
	in rank 3. Then in light of the preceding remarks, we only need the existence of polytopes for
	every symmetry type graph in rank 3. This is shown in \cite[Sec. 5]{k-orbit-maps}.
	\end{proof}

%	Theorem \ref{t:2ToTheP} can also be seen to be true by using a trivial rank $n+1$ extension of the rank $n$ polytope $\mathcal{P}$ consisting of two vertices, and an element $F'$ of rank $j+1$ for every element $F$ of $\mathcal{P}$ of rank $j$. Here $F' \le G'$ if and only if $F \le G$, and any element $F'$ is incident to the two vertices of the extension. Clearly this extension has the graph described in Theorem \ref{t:2ToTheP} as its symmetry type graph; however, this extension is not helpful when restricting to polytopes with no digonal sections.

	The previous discussion leads to the following reinterpretation of Problem \ref{p:allowableTypeGraph}.

	\begin{problem}
	Are there constructions of $(n+1)$-polytopes $\mathcal{P}$ from $n$-polytopes where the symmetry type graph of $\mathcal{P}$ is different from those in Theorem \ref{t:2ToTheP} and from their duals?
	\end{problem}

	We can specialize Problem \ref{p:allowableTypeGraph} further by asking which symmetry type graphs occur
	among convex polytopes. In fact, this is really two problems, since we may look at the combinatorial
	symmetry type graph or the geometric one. Some results in \cite[Sec. 2.4]{kolya-thesis} provide additional
	necessary conditions for the geometric symmetry type graph of a convex polytope.
	
	\begin{problem}
	Which (combinatorial) symmetry type graphs occur among \emph{convex} polytopes? Which geometric symmetry type graphs occur?
	\end{problem}
	
\subsection{Amalgamations and extensions}

	It is a standard technique to construct an $n$-polytope by carefully putting together a family of $(n-1)$-polytopes. In the simplest case, the $(n-1)$-polytopes are all isomorphic. The problem of determining the possible $n$-polytopes
	we can assemble from copies of a single $(n-1)$-polytope is surprisingly deep, even in the convex setting.
	For example, there are only 8 convex polyhedra (up to similarity) whose faces are all equilateral triangles
	(see for example \cite{Deltahedra}). In general, if $\calP$ is an abstract polytope, we say that $\calR$ is
	an \emph{extension} of $\calP$ if every facet of $\calR$ is isomorphic to $\calP$.
	
	We may also require that we arrange the $(n-1)$-polytopes in the same way around every vertex, so that
	the vertex-figures will be isomorphic. This is of course much more restrictive, even if we do
	not prescribe the vertex-figure in advance. For example, the only convex polyhedra built out of
	equilateral triangles meeting the same way at each vertex are the tetrahedron (with triangular
	vertex-figures), the octahedron (with quadrangular vertex-figures), and the icosahedron
	(with pentagonal vertex-figures). In general, if $\calP$ and $\calQ$ are abstract polytopes, we say
	that $\calR$ is an \emph{amalgamation of $\calP$ and $\calQ$} if every facet of $\calR$ is isomorphic to $\calP$
	and every vertex-figure of $\calR$ is isomorphic to $\calQ$. The set of all amalgamations of
	$\calP$ and $\calQ$ is denoted by $\langle \calP, \calQ \rangle$.
	
	In order for $\langle \calP, \calQ \rangle$ to be nonempty, the polytopes $\calP$ and
	$\calQ$ must be compatible: every vertex-figure of $\calP$ must be isomorphic to a facet of $\calQ$.
	In general, this condition does not suffice. For example, even though hemi-cubes have triangular vertex-figures
	and icosahedra have triangular facets, there are no $4$-polytopes whose facets are hemi-cubes and whose vertex-figures are icosahedra (see \cite[Theorem 3.6]{LocProjective}). 
	
	One of the driving forces behind the amalgamation problem was Gr\"unbaum's problem, posed in
	\cite{GrunbaumToroidal}, of classifying the \emph{locally toroidal}
	polytopes. These are polytopes such that their facets and vertex-figures are either
	toroidal or spherical, but not both spherical. Motivated by this problem, amalgamations of regular and chiral polytopes have been previously addressed (see for example \cite[Chapter 4]{arp} and \cite{Census}). 
	See also \cite[Sec. 4]{polytope-problems} and \cite[Probs. 31-32]{chiral-problems} for an overview of what is known about amalgamations of regular and chiral polytopes. Amalgamations of other classes of polytopes
	remain largely unexplored.

	The vertex-figures of any polytope $\calR$ in $\langle \calP, \calQ \rangle$ are all isomorphic, but the vertex-figures of $\calP$ may not need to be all isomorphic. In principle there may be (for example) $4$-polytopes whose facets are all rhombic dodecahedra and whose vertex-figures are cuboctahedra. (The {\em cuboctahedron} is a polyhedron with 6 squares and 8 triangles such that two squares and two triangles meet at each of its 12 vertices in an alternating manner, as in Figure \ref{f:cuboctahedron} (a). The {\em rhombic dodecahedron} is the dual of the cuboctahedron; it has 12 quadrilateral faces meeting three of them on 8 of its vertices, and four of them in the remaining 6 vertices, as in Figure \ref{f:cuboctahedron} (b).) In this way the vertex-figures at some vertices of a given facet would correspond to triangles on their cuboctahedral vertex-figures in $\calR$, whereas the remaining vertex-figures at the same facet would correspond to squares. The amalgamation problem so far has been considered only in the context of regular and chiral polytopes, and to the authors' knowledge, no amalgamation has been found where the facets of the vertex-figures
	are of two or more types.
	
\begin{figure}
\begin{center}
\includegraphics{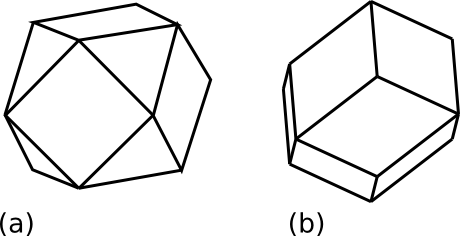}
\caption{Cuboctahedron and rhombic dodecahedron}\label{f:cuboctahedron}
\end{center}
\end{figure}

	\begin{problem}
	Are there polytopes $\calP$ and $\calQ$ such that $\langle \calP, \calQ \rangle$ is nonempty
	and $\calP$ has at least two non-isomorphic vertex-figures?
	\end{problem}

	An amalgamation may also be understood as a rooted object. That is, given a flag orbit $\mathcal{O}_1$ of $\calP$ and a flag orbit $\mathcal{O}_2$ of $\calQ$ we can consider only amalgamations $\calR$ of $\calP$ and $\calQ$ where each flag of $\calR$ that is in $\mathcal{O}_1$ when restricted to its facet, is in $\mathcal{O}_2$ when restricted to the vertex-figure; and where each flag of $\calR$ that is in $\mathcal{O}_2$ when restricted to its vertex-figure, it is in $\mathcal{O}_1$ when restricted to the facet. This was implicitly done in \cite{Census}, where polytopes $\calR$ in $\langle \{4,4\}_{(a,b)}, \{4,4\}_{(c,d)} \rangle$ and in $\langle \{4,4\}_{(a,b)}, \{4,4\}_{(d,c)} \rangle$ correspond to the two choices of flag orbits of the vertex-figure when preestablishing a flag of $\calR$ containing a given flag of the facet $\{4,4\}_{(a,b)}$.

	An example of an unrooted amalgamation of two polytopes can be found in \cite{UniformTilingsR3}, where tiling $\#12$ is a vertex-transitive nine-orbit tessellation of Euclidean space by triangular prisms. The vertex-figures are all isomorphic to a nine-orbit facet-transitive convex polyhedron with isosceles triangles as faces, so each flag orbit of the tiling corresponds to a flag orbit of the vertex-figures. The facets, being triangular prisms, have three flag orbits, and since the tiling itself has nine flag orbits, it follows that some flags that correspond to the same flag orbit in any given prism must belong to flags in distinct orbits of the tiling. So two flags that induce the same flag orbit of a facet may be paired with distinct flag orbits of a vertex-figure.

 % The tiling $\#12$ is transitive on facets, implying that the 9 flag orbits must appear on each triangular prism. Since the triangular prism has 3 flag orbits, some flags in the same orbit of any given prism belong to flags in distinct orbits of the tiling. Now, each flag orbit of the tiling corresponds to a flag orbit of the vertex-figures, so three flags in distinct orbits of the vertex-figure are contained in flags of the tiling having the same restriction to the facets of the tiling.

	%Assume that $\calR$ is an amalgamation of $\calP$ rooted at $\Phi$ and $\calQ$ rooted at $\Psi$. If $v_0$ is the vertex in $\Phi$ and $f_0$ is the facet in $\Psi$ then it must happen that the vertex figure of $\calP$ at $v_0$ rooted at $\Phi \setminus \{v_0\}$ must be isomorphic (as rooted polytope) to the facet $f_0$ of $\calQ$ rooted at $\Psi \setminus \{f_0\}$. It may be the case that $\calP$ and $\calQ$ have amalgamations when rooted on $\Phi$ and $\Psi$, but not on other pairs of flags.

	Whenever $\calP$ and $\calQ$ are both regular and there is a regular amalgamation of $\calP$ and $\calQ$, then there is a {\em universal regular amalgamation} of $\calP$ and $\calQ$ (denoted $\{\calP, \calQ\}$) which covers any other regular amalgamation of $\calP$ and $\calQ$ (see \cite{amalgamation}). A similar result holds when $\calP$ or $\calQ$ is chiral (see \cite{chiral}). When both $\calP$ and $\calQ$ are chiral, the universal amalgamation is in fact rooted.
	% the amalgamations in the definition of universal amalgamation are in fact rooted.
	For example, the universal amalgamation of $\{4,4\}_{(1,3)}$ with $\{4,4\}_{(1,3)}$ has 24 vertices, whereas the universal amalgamation of $\{4,4\}_{(3,1)}$ with $\{4,4\}_{(1,3)}$ has 50 vertices \cite{Census}.

	The existence of universal amalgamations of regular and chiral polytopes relies on the structures of their automorphism groups. It is not known if universal amalgamations exist in general.

	\begin{problem}
	If there is a rooted amalgamation of the polytopes $\calP$ and $\calQ$, then is there always a universal rooted 
	amalgamation of $\calP$ and $\calQ$?
	\end{problem}

	The following proposition provides a lower bound on the number of flag orbits of an amalgamation of $\calP$ and $\calQ$.

	\begin{proposition}\label{p:lowBoundOrbitAmalgam}
	Let $\calP$ be a $k_1$-orbit polytope and $\calQ$ be a $k_2$-orbit polytope. If $\calR$ is a $k_3$-orbit amalgamation of $\calP$ and $\calQ$ then $lcm(k_1,k_2)$ divides $k_3$.
	\end{proposition}

	\begin{proof}
	Let $F$ be a facet of $\calR$. We claim that the number of flag orbits of $\calR$ containing $F$ is a multiple of $k_1$. This follows from the fact that the set of flags containing $F$ in any given flag orbit of $\calR$ consists of the images of any of them under $\Stab_{\Gamma(\calR)}(F)$. The latter group must be isomorphic to a subgroup of $\Gamma(F)$ of some index $m$. This implies that there are $mk_1$ flag orbits of $\calR$ containing flags that contain $F$.

	If a flag orbit contains a flag $\Phi$ that contains a facet $F$ and a flag $\Psi$ that contains a facet $G$ then there must be an automorphism of $\calR$ mapping $F$ to $G$; in particular, there is an isomorphism from $F$ to $G$ mapping the restriction of $\Phi$ to the restriction of $\Psi$.
	It follows that if $F$ and $G$ are contained on some flags in any given flag orbit then the flag orbits of $\calR$ involved in $F$ are the same as those involved in $G$.

	Naturally, $\Gamma(\calR)$ induces a partition on the facets of $\calR$ where two facets are equivalent if and only if there is an automorphism mapping one on the other. The discussion above implies that the number $k_3$ of flag orbits is the sum over the parts of the partition of the number of flag orbits represented on each part; moreover, this is the same as the sum of the number of flag orbits represented on one facet on each part. Each term of this sum is a multiple of $k_1$. Hence $k_1$ divides $k_3$.

	A dual argument establishes that $k_2$ divides $k_3$.
	\end{proof}
	%\textbf{I'm a little confused -- are you using $m$ in two different ways accidentally? $\calR$ is supposed
	%to be an $m$-orbit polytope, but in the first paragraph, you have that $\calR$ has $mk_1$ different flag
	%orbits that contain (flags that contain) $F$.}
	
	The lower bound on the number of orbits of an amalgamation in Proposition \ref{p:lowBoundOrbitAmalgam} can certainly be achieved. The simplest cases are when $\calR$ is regular (and hence its facets and vertex-figures are also regular) or when it is chiral with chiral facets and vertex-figures. 

	There is no upper bound in general on the number of flag orbits of amalgamations of two given polytopes. For example, the number of orbits of a polyhedron with square faces and square vertex-figures can be arbitrarily large (see \cite{Bottle}); this is also the case for polyhedra with cubical faces and octahedral vertex-figures (see \cite{Didicosm} and \cite{Helicosms}). Whenever the universal amalgamation of $\calP$ and $\calQ$ exists and is finite, then the number of flag orbits of the amalgamations of these two polytopes is bounded by the number of flags of the universal
	amalgamation, but not necessarily by the number of flag \emph{orbits} of the universal amalgamation. It is not known 
	whether this is the only case when a class of amalgamations has a bound on the number of flag orbits.

	\begin{problem}
	Classify all pairs of polytopes $\calP$ and $\calQ$ such that $\langle \calP, \calQ \rangle$ is non-empty and the number of flag orbits of polytopes in $\langle \calP, \calQ \rangle$ is bounded.
	\end{problem}

	Let us return now to the problem of extending a polytope $\calP$ without prescribing the vertex-figures.
	Some work has been done on extensions of regular and chiral polytopes, including
	regular extensions of regular polytopes (see \cite{pellicer-extension}), chiral extensions of regular polytopes
	(see \cite{high-rank-chiral}), and chiral extensions of chiral polytopes (see \cite{chiral-ext2, chiral-ext}).
	When $\calP$ is regular, then there is a \emph{universal extension} of $\calP$ that covers every other
	regular extension of $\calP$ \cite[Thm. 4D4]{arp}. Similarly, if $\calP$ is chiral with regular facets, 
	then there is a universal chiral extension of $\calP$ that covers all other chiral extensions of
	$\calP$ \cite{chiral-ext}. 
	
	\begin{problem}
	Which $k$-orbit polytopes $\calP$ have a universal $k$-orbit extension that covers all other $k$-orbit
	extensions of $\calP$?
	\end{problem}

	Rather that looking at extensions or amalgamations with a single type of facet, we can allow
	ourselves to use two different polytopes as facets. For example, in \cite{semiregular}, the
	authors arrange two regular polytopes $\calP$ and $\calQ$ in an alternating fashion to create
	vertex-transitive polytopes. Again, the construction is essentially algebraic, using the particular
	structure of the automorphism group of a regular polytope. Is there a useful generalization to
	$k$-orbit polytopes?
	
	\begin{problem}
	Given two polytopes $\calP$ and $\calQ$, when is there a polytope whose facets
	are each isomorphic to $\calP$ or $\calQ$? When is there a finite such polytope?
	\end{problem}

%{\bf Note for the future:} I did check the papers

%`On the size of equifacetted semi-regular polytopes' by Egon, Asia and Tomo,

%`Semiregular polytopes and amalgamated C-groups' by Egon and Barry

%and I found no construction relevant for this section. In the first one they use $2^{\mathcal{P}}$ and in the second they use a free amalgamation thing that yields only regular and $2$-orbit polytopes in class $2_{\{0, \dots, n-1\}}$.

%\bibliographystyle{amsplain}
%\bibliography{gabe}

%\end{document}

\section{Realizations}
The geometric origin of the idea of a polytope suggests that we study possible geometric descriptions of abstract polytopes. With this in mind, a {\em realization} of a polytope $\calP$ in a geometric space $S$ is defined in \cite{realizations-1} as a function from the vertex set of $\calP$ to some discrete subset of $S$, together with some instructions of how to interpret faces of higher ranks. Essentially, every face $F$ of rank $m+1$ is associated to the set of images of the $m$-faces contained in $F$.

Sometimes it is convenient to give further meaning to faces of rank 1 or greater. Whenever there is a unique well-defined line segment between any pair of points in $S$, then the edge between vertices $u$ and $v$ can be associated to the line segment between the points in $S$ associated to $u$ and $v$. This is the case in (real) Euclidean $d$-space $\mathbb{E}^d$ and (real) hyperbolic $d$-space $\mathbb{H}^d$. However, in (real) projective $d$-space $\mathbb{P}^d$ there are precisely two well-defined line segments between any pair of distinct points. In that case we may require the realization to also indicate the line segment associated to each edge (see \cite{ProjectivePoyhedraArocha}).

When convenient we may abuse notation and use $F$ to denote not just a face, but also its image
under the realization.

A realization is said to be {\em faithful} whenever for every $i$, the realization induces an injective mapping from the set of $i$-faces. Faithful realizations of abstract polytopes are often called {\em geometric polytopes}.

A \emph{symmetry} of a geometric polytope $\mathcal{K}$ is an isometry of $S$ that preserves $\mathcal{K}$. %We will assume, as it is often done, that the identity is the only symmetry of $\mathcal{K}$ that fixes every face of $\mathcal{K}$. This can be understood as restricting $S$ to the smallest subspace containing $\mathcal{K}$, or else by considering instead the quotient of $S$ by the group of symmetries fixing every face of $\mathcal{K}$.
A realization $\mathcal{K}$ of an abstract polytope $\calP$ is said to be {\em symmetric} whenever, for every automorphism $\gamma$ of $\calP$, there is a symmetry of $\mathcal{K}$ that acts like $\gamma$ on the faces of $\mathcal{K}$. Clearly, $\Gamma(\calP)$ is isomorphic to the quotient of the symmetry group of any symmetric faithful realization of $\calP$ by the stabilizer of any of the flags. 

When convenient we shall abuse notation and denote by $\calP$ both the abstract polytope and its realization. Then the symmetry group is denoted by $G(\calP)$ to distinguish it from the automorphism group $\Gamma(\calP)$.

A geometric polytope $\calP$ is said to be {\em geometrically regular} if $G(\calP)$ acts transitively on the flags. If $\calP$ is geometrically regular then it is also combinatorially regular. A polytope is {\em combinatorially $k$-orbit} (resp. {\em combinatorially chiral}) whenever $G(\calP)$ induces $k$ orbits on the flags (resp. $2$ orbits on the flags with adjacent flags in distinct orbits). A geometrically $k$-orbit polytope must be a combinatorially $m$-orbit polytope for some $m$ dividing $k$.

The problems on realizations have been divided into two sections. In the first, we fix the space and explore the possible geometric polytopes that live there. In the second, we fix the polytope and ask where it can be realized and with what possible characteristics.
% One of them explores the possible geometric polyhedra on a given space, while the other one discusses the possible realizations of a given polytope on different spaces.

\subsection{Realizations in a given space}

When searching for all regular geometric $n$-polytopes ($n \ge 3$) in a geometric space, the first published complete list is that of the 48 regular polyhedra in $\mathbb{E}^3$. It consists of 18 finite polyhedra, 6 planar polyhedra (that are usually kept in this list rather than constituting a complete list of regular polyhedra in $\mathbb{E}^2$), 12 blended infinite polyhedra and 12 pure infinite polyhedra (see \cite{DressI}, \cite{DressII} and \cite{OrdinarySpace}).

Blended polyhedra are constructed from two orthogonal components, and depend on one real parameter that rescales one of the components while fixing the dimensions of the other. Pure polyhedra are the ones that are not blended. The 36 polyhedra that are not blended are unique up to similarity.

A lot of progress towards the classification of regular $n$-polytopes in $\mathbb{E}^d$ has been made in the last 20 years, mainly through the work of Peter McMullen. Most of the progress has centered on polytopes of
\emph{full rank} and \emph{nearly full rank}, which will be defined shortly.

A fair amount is also known about chiral geometric polytopes. There are no convex chiral polytopes and no chiral tessellations of Euclidean spaces (see for example \cite{DressEgon} and \cite{EuclideanSpaceForms}). The first chiral (geometric) polytopes were found only in 2005, when the complete classification of chiral polyhedra in $\mathbb{E}^3$ was given in \cite{ChiralRealSpaceI} and \cite{ChiralRealSpaceII}. They are all infinite and can be organized in six families, each of them depending on one parameter (up to similarity). Polyhedra in three of the families are combinatorially chiral and have finite faces; the parameter on these families takes only rational numbers. Polyhedra in each of the remaining three families are all combinatorially isomorphic to some regular polytope with infinite faces; the parameter runs over the real numbers and it can be interpreted as a chiral continuous motion of the infinite regular polyhedron.

There are no regular or chiral finite $(d+1)$-polytopes or infinite $(d+2)$-polytopes in $\mathbb{E}^d$ (see \cite[Theorems 5B20, 5C3]{arp} and \cite[Theorem 3.1, Proposition 11.1]{FullRank}). Hence finite regular or chiral $d$-polytopes and infinite $(d+1)$-polytopes in $\mathbb{E}^d$ are called {\em polytopes of full rank}.

Regular polytopes of full rank were classified in \cite[Theorem 11.2]{FullRank}, where it is also stated that there are no chiral polytopes of full rank. This claim turned out to be false; there are both finite chiral $4$-polytopes in $\mathbb{E}^4$ \cite{finite-chiral-4-poly} and infinite $4$-polytopes in $\mathbb{E}^3$ \cite{chiral-4-poly-r3}. It is natural now to ask for the classification of full rank chiral polytopes.

\begin{problem}\label{p:ChiralFullRank}
Describe all chiral polytopes of full rank.
\end{problem}

Regular and chiral finite $(d-1)$-polytopes and infinite $d$-polytopes in $\mathbb{E}^d$ are naturally called of {\em nearly full rank}. Regular polytopes of nearly full rank were classified in \cite{FourDimensionalFinite}, \cite{FourDinensionalInfinite} \cite{NearlyFullRank} and \cite{AddendumPeter}. The following is a more challenging problem than Problem \ref{p:ChiralFullRank}.

\begin{problem}
Describe all chiral polytopes of nearly full rank.
\end{problem}

% Setting aside regularity and chirality, finite polyhedra and infinite $4$-polytopes may be realized in $\mathbb{E}^2$. For example, consider the structure $\calQ$ consisting of the following three $3$-polyhedra. Face $F_1$ is the regular tessellation $\{4,4\}$ by squares of $\mathbb{E}^2$. Given the natural bipartition $\{X_2,X_3\}$ of the squares of $F_1$, we define $F_2$ (resp. $F_3$) as the polyhedron with the same sets of vertices and edges of $F_1$ whose $2$-faces are all linear apeirogons (tessellations by edges of a straight line) in $F_1$ and $X_2$ (resp. $X_3$). It is not hard to verify that $\calQ$ is indeed a $4$-polytope in $\mathbb{E}^2$. Furthermore, $\calQ$ has $3$ flag orbits under $G(\calQ)$. This shows a symmetric faithful embedding of a three-orbit $4$-polytope in $\mathbb{E}^2$. The vertex-figures of $\calQ$ are combinatorially isomorphic to the hemicube (or equivalently, the Petrial of the tetrahedron), giving a planar three-orbit realization of a finite regular polyhedron.

If we set aside regularity and chirality, then there are finite $(d+1)$-polytopes and infinite $(d+2)$-polytopes that may be realized in $\mathbb{E}^d$. For example, consider the structure $\calQ$ consisting of the following three $3$-polyhedra. Face $F_1$ is the regular tessellation $\{4,4\}$ by squares of $\mathbb{E}^2$. Now color the squares in a checkerboard pattern, and let $F_2$ consist of the same vertices and edges as $F_1$, all of the black squares, and all linear apeirogons. That is, in addition to the black squares, $F_2$ has faces that are tessellations of the vertical and horizontal lines by edges. Let $F_3$ be the same as $F_2$ but with the white squares rather than the black ones. It is not hard to verify that $\calQ$ is a $4$-polytope in $\mathbb{E}^2$. Furthermore, $\calQ$ has $3$ flag orbits under $G(\calQ)$. This shows a symmetric faithful embedding of a three-orbit $4$-polytope in $\mathbb{E}^2$. The vertex-figures of $\calQ$ are combinatoially isomorphic to the hemicube (or equivalently, the Petrial of the tetrahedron), giving a planar three-orbit realization of a finite regular polyhedron.

\begin{problem}\label{p:max-rank}
For each symmetry type graph $T$, determine the maximum $n$ such that finite (or infinite) $n$-polytopes admit a faithful symmetric realization in $\mathbb{E}^d$ with symmetry type graph $T$.
\end{problem}

Given a symmetry type graph, we naturally define any $n$-polytope in $\mathbb{E}^d$ to be of full rank if there are no
$(n+1)$-polytopes in $\mathbb{E}^d$ with the same symmetry type graph. The following problem seems to be the
next natural step after Problem~\ref{p:ChiralFullRank}.

% The next problem is natural and seems to be also feasible.

\begin{problem}
Describe all two-orbit polytopes of full rank.
\end{problem}

In \cite{Index2-I} and \cite{Index2-II} all two-orbit realizations in $\mathbb{E}^3$ of finite regular polyhedra were found. They are still highly symmetric interesting structures. When looking for $k$-orbit realizations of regular polytopes it might be a good idea to bound $k$ so that the resulting structures are all still highly symmetric and the classification remains tractable.

\begin{problem}
Find all $k$-orbit realizations in $\mathbb{E}^d$ of regular $d$-polytopes for $k \le d$.
\end{problem}

Convexity represents a serious restriction for highly symmetric structures. It is well-known by now that there are only $5$ regular convex polyhedra, $6$ regular convex $4$-polytopes and $3$ regular convex $d$-polytopes for each $d \ge 5$. Convex $k$-orbit polytopes are also very restricted as shown by the following theorem (\cite[Chapter VIII, Theorem 1]{kolya-thesis}).

\begin{theorem}
For any integer $k>1$ there is $N_k \in \mathbb{N}$ such that (geometrically) $k$-orbit convex polytopes exist only in fewer than $N_k$ dimensions.
\end{theorem}

As metioned in \cite[Chapter VIII]{kolya-thesis}, $N_2 = 4$ and $N_3 = 9$; furthermore, $k+2 \le N_k \le 2^{k-3} 9$.

\begin{problem}\label{p:Convexnk}
Determine $N_k$ for every $k$ or improve the bounds for it.
\end{problem}

More generally, we have the following problems.

\begin{problem} \label{p:convex1}
For what values of $k$ and $d$ is there a convex $d$-polytope with $k$ combinatorial flag orbits?
\end{problem}

\begin{problem} \label{p:convex2}
For what values of $k$ and $d$ is there a convex $d$-polytope with $k$ geometric flag orbits?
\end{problem}
		
Strong restrictions also arise when realizing $k$-orbit $d$-polytopes as convex $d$-polytopes. Theorem 1 of \cite{combinatorially-two-orbit} states that if a two-orbit $d$-polytope $\calP$ is realized as a convex $d$-polytope, then $\calP$ must be isomorphic to a geometrically two-orbit convex polytope. The classification in \cite{two-orbit-convex} then implies that $\calP$ is either the cuboctahedron, the icosidodecahedron, or the dual of one of those two. Does a similar result hold for other $k$-orbit polytopes?

\begin{problem}\label{p:Convexkorbit}
For what values of $k$ and $d$ is it true that, if a $k$-orbit $d$-polytope $\calP$ is realized as a convex $d$-polytope, then $\calP$ must be isomorphic to a geometrically $k$-orbit convex polytope?
\end{problem}

Problems \ref{p:Convexnk} and \ref{p:Convexkorbit} are restricted to finite polytopes. One can rephrase them replacing ``convex polytopes'' by ``locally finite face-to-face tilings by convex tiles'' to formulate similar problems for infinite polytopes with finite faces. More information about this setting can be found in \cite{kolya-thesis}, \cite{two-orbit-convex} and \cite{combinatorially-two-orbit}.

	%There is also a similar result
	%for normal tilings (tilings with a uniform bound on the size of the tiles).
	
% Other interesting problems of non-regular realizations in Euclidean spaces have been of interest, like realization whose symmetry group acts transitively on some chains. We believe that topic to be already too large to comprise in this work and we refer the reader to other texts.

Describing all geometrically vertex-transitive polyhedra with geometrically regular faces in $\mathbb{E}^3$ has already been a quite challenging problem that has not been solved even for the finite case (see \cite{CoxeterUniform} for a description of those with planar faces). Partial progress on the problem when the restriction of finiteness or of regular faces is dropped can be found in \cite{AbigailWilliamsThesis} and \cite{Undine}. These topics in higher ranks seem unexplored so far.

%Realizations on Euclidean spaces have some restrictions. For example, the hemioctahedron (see Figure \ref{f:hemiocta}) cannot be realized faithfully on any Euclidean space since any two vertices are joined by two distinct edges. This polyhedron admits a realization on $\mathbb{P}^2$ thanks to the two line segments between any pair of vertices.

%\begin{figure}
%\begin{center}
%\includegraphics[width=4cm]{hemiocta.eps}
%\caption{Hemioctahedron on the projective plane; opposite points of the outer circle are identified}\label{f:hemiocta}
%\end{center}
%\end{figure}

Moving away from Euclidean spaces the most natural spaces to explore are $\mathbb{P}^d$ and $\mathbb{H}^d$.

As mentioned in \cite{FourDimensionalFinite}, $d$-polytopes in $d$-projective space can be obtained from finite $d$-polytopes in $\mathbb{E}^{d+1}$ whose vertices lie on some sphere. Consequently, the classification of $k$-orbit $n$-polytopes in $\mathbb{P}^d$ can be obtained from the classification of finite $k$-orbit $n$-polytopes in $\mathbb{E}^{k+1}$ whose vertices live all on some sphere. Some results on realizations in $\mathbb{P}^3$ can be found in \cite{ProjectivePoyhedraArocha}, \cite{RoliProjectivePolyhedra}, \cite{finite-chiral-4-poly} and \cite{Realizations-toroids}.

Classifying regular polytopes in $\mathbb{H}^3$ seems much harder than classifying them in $\mathbb{E}^3$ (see \cite[Section 13]{FullRank} for an argument toward some difficulty in rank 4). Besides the work on the regular tilings by convex polyhedra (see for example \cite[Section 6J]{arp}) and the works \cite{CyrilGarner} and \cite{Hyperbolic-polyhedra}, to the authors' knowledge there has been no progress on this interesting problem. Classifying realizations of $k$-orbit polytopes in hyperbolic spaces seems prohibitively difficult at the moment. Perhaps the following problems would be tractable and would provide a good start to the theory of hyperbolic realizations of $k$-orbit polytopes.

\begin{problem}
Classify all two-orbit polyhedra in $\mathbb{H}^3$.
\end{problem}

\begin{problem}
Classify all two-orbit $4$-polytopes in $\mathbb{H}^3$.
\end{problem}

We suggest that readers eager to explore different spaces for realizations should consider quotients of $\mathbb{S}^d$, of $\mathbb{E}^d$ or of $\mathbb{H}^d$ (see \cite[Chapter 6]{arp}).

\subsection{Realizations of a given polytope}

Regular realizations of a given regular polytope in Euclidean spaces have been extensively studied. The space of Euclidean regular realizations of a given regular polytope has the structure of a convex cone. The realization cones of several regular polytopes have been described (see \cite{realizations-1}, \cite{realizations-3}, \cite{realizations-4}, \cite{realizations-2}).

Much less is known about other kind of realizations. Some ideas used in the study of regular realizations of a regular polytope may be adapted to the more general setting of symmetric realizations of a $k$-orbit polytope. The only examples so far of different symmetric realizations of non-regular polytopes can be found in \cite{Realizations-toroids}, where symmetric realizations of any toroid $\calP$ with type $\{4,4\}$ are given in terms of symmetric realizations of a regular toroid where $\calP$ can be embedded.

\begin{problem}\label{p:k-orb-real}
Develop a theory of $k$-orbit realizations of $k$-orbit polytopes.
\end{problem}

Here is a particularly appealing special case of Problem~\ref{p:k-orb-real}.

\begin{problem}
Which $k$-orbit $d$-polytopes have a faithful realization as a geometrically $k$-orbit polytope in $\mathbb{E}^d$?
\end{problem}

We should expect to need new techniques when studying non-symmetric realizations. In particular, some interesting behavior has appeared with two-orbit polyhedra that are combinatorially regular, namely that several of them admit a continuous movement while preserving the symmetry group (as an abstract group) at all times and the edge length of some edges (in some cases of all edges).
This is the case of all two-orbit face-to-face tessellations of Euclidean spaces by convex tiles that are combinatorially regular \cite{two-orbit-convex}, the infinite families of regular polyhedra of index $2$ \cite{Index2-I}, and all chiral realizations of regular polyhedra in $\mathbb{E}^3$ \cite{CombStructureSchulChiral}. In each case the movement has a different nature as that of the blended polyhedra in Euclidean space. However, the regular polyhedra of index $2$ in \cite{Index2-II} admit no such movement.

This may only be one of many different aspects of the study of non-symmetric realizations with respect to that of symmetric realizations.

\begin{problem}\label{p:kmrealizations}
Develop a theory of $k$-orbit realizations of $m$-orbit polytopes.
\end{problem}

A first step in the direction of Problem \ref{p:kmrealizations} is to restrict to the case $m=1$.

\begin{problem}
Develop a theory of $k$-orbit realizations of regular polytopes.
\end{problem}

Euclidean spaces seem the most natural spaces to choose when studying realizations, in part due to the decomposition of its group of isometries as a semidirect product of the translation subgroup and the subgroup of linear isometries.

Hyperbolic spaces admit several symmetric realizations of regular polytopes that cannot be realized in the Euclidean space of the same dimension, like the regular tessellations $\{p,q\}$ with $1/p + 1/q < 1/2$ of $\mathbb{H}^2$ \cite[Chapter 6J]{arp}. %On the other hand, regular tessellations of hyperbolic spaces by convex compact tiles occur only in $\mathbb{H}^2$ and $\mathbb{H}^3$.
However, to the authors' knowledge there is no systematic work on the study of all hyperbolic realizations of a given polytope.

\begin{problem}
Develop a theory of realizations of regular polytopes in hyperbolic spaces.
\end{problem}

\section{Acknowledgements}
The second author was supported by PAPIIT-UNAM under project grant IN101615.

\bibliographystyle{amsplain}
\bibliography{gabe}

\end{document}